\documentclass[reqno]{amsart}

\usepackage{amsmath,amssymb,amsthm,amsfonts,mathrsfs}
\usepackage{fullpage}
\usepackage{xcolor}
\usepackage[colorlinks=true,
linkcolor=blue,
anchorcolor=blue,
citecolor=red
]{hyperref}
\allowdisplaybreaks % allow equations break between two pages
\newtheorem{theorem}{Theorem}[section]
\newtheorem{lemma}[theorem]{Lemma}
\newtheorem{corollary}[theorem]{Corollary}

\newtheorem*{conjecture*}{Conjecture}
\theoremstyle{definition}

\theoremstyle{remark}

\newtheorem*{remark*}{remark}
\usepackage{colonequals}

\author{Runbo Li}
\address{The High School Affiliated to Renmin University of China International Curriculum Center, Beijing 100080, People's Republic of China}
\email{carey.lee.0433@gmail.com}

\makeatletter
\@namedef{subjclassname@2020}{\textup{}2020 Mathematics Subject Classification}
\makeatother

\title[]{Hybrid estimation of single exponential sums, exceptional characters and primes in short intervals}
\subjclass[2020]{11M06, 11N05, 11N37} 
\keywords{primes, short intervals, exponential sums, exceptional characters}

\begin{document}
	
\begin{abstract}
We provide a new hybrid estimation of single exponential sums, combining Van der Corput, Huxley and Bourgain's result. We also focus on primes in short intervals $(x-x^{\alpha},x]$ under the assumption of the existence of exceptional Dirichlet characters and get a small improvement of a 2004 result of Friedlander and Iwaniec. By using our new estimation of exponential sums, we extend the previous admissible range $0.4937 \leqslant \alpha \leqslant 1$ to $0.4923 \leqslant \alpha \leqslant 1$.
\end{abstract}

\maketitle

%\hfill \textit{} 

\section{Introduction}
The famous Riemann Hypothesis is equivalent to the asymptotic formulas
\begin{equation}
\psi(x):=\sum_{n \leqslant x} \Lambda(n)=x+O\left(x^{\frac{1}{2}+\varepsilon}\right) \quad \operatorname{and} \quad \pi(x)=\frac{x}{\log x}+O\left(x^{\frac{1}{2}+\varepsilon}\right),
\end{equation}
where $\Lambda(n)$ denotes the von Mangoldt function. These imply that
\begin{equation}
\psi(x)-\psi(x-x^{\alpha})=x^{\alpha}+O\left(x^{\frac{1}{2}+\varepsilon}\right) \quad \operatorname{and} \quad \pi(x)-\pi(x-x^{\alpha})=\frac{x^{\alpha}}{\log x}+O\left(x^{\frac{1}{2}+\varepsilon}\right)
\end{equation}
for $\frac{1}{2}+\varepsilon <\alpha \leqslant 1$. Clearly this shows that there is always a prime number in the short interval $(x-x^{\frac{1}{2}+\varepsilon},x]$. Unfortunately, we can't prove (2) unconditionally. Now the best unconditional result is due to Baker, Harman and Pintz. In \cite{BHP} they showed that $0.525 \leqslant \alpha \leqslant 1$ is admissible. But even we assume the Riemann Hypothesis, we can't extend the range $\frac{1}{2}+\varepsilon <\alpha \leqslant 1$ anymore.

In 2004, under the assumption of the existence of exceptional Dirichlet characters, Friedlander and Iwaniec \cite{FI2004} first extended the range of $\alpha$ to some numbers below $\frac{1}{2}$. Actually they proved the following theorem:
\begin{theorem}\label{FI0.4937}
([\cite{FI2004}, Theorem 1.1]).
Let $\chi=\chi_{D}$ denotes the real primitive character of conductor D, $x \geqslant D^{r}$ with $r=18289$ and $\frac{39}{79} \leqslant \alpha \leqslant 1$. Then we have
$$
\psi(x)-\psi(x-x^{\alpha})=x^{\alpha} \left\{1+O\left(L(1, \chi)(\log x)^{r^{r}}\right)\right\}
$$
and
$$
\pi(x)-\pi(x-x^{\alpha})=\frac{x^{\alpha}}{\log x}\left\{1+O\left(L(1, \chi)(\log x)^{r^{r}}\right)\right\},
$$
where $L(s, \chi)$ is the Dirichlet L-function.
\end{theorem}
\noindent In fact, their result strongly depends on a deep result involving product of three Dirichlet $L$-series by themselves (\cite{FI2005}, with only classical Van der Corput method):
\begin{theorem}\label{FI-L}([\cite{FI2005}, Theorem 4.2]).
Let $\chi_{j}\left(\bmod D_{j}\right)$ denote primitive characters and $D=D_{1} D_{2} D_{3}$. For any $x \geqslant 1$, we have
\begin{align}
\nonumber \Delta:= & \sum_{n_{1} n_{2} n_{3} \leqslant x} \chi_{1}\left(n_{1}\right) \chi_{2}\left(n_{2}\right) \chi_{3}\left(n_{3}\right)-\operatorname{Res}_{s=1} \left(L\left(s, \chi_{1}\right) L\left(s, \chi_{2}\right) L\left(s, \chi_{3}\right) \frac{x^s}{s}\right) \\
\nonumber \ll & D^{\frac{38}{75}} x^{\frac{37}{75}+\varepsilon}.
\end{align}
\end{theorem}

In 2017, based on his previous work on higher order derivative tests for exponential sums \cite{Nowak2012}, Nowak \cite{Nowak1} \cite{Nowak2} combined the classic Van der Corput method with Huxley's "Discrete Hardy--Littlewood method" (see \cite{HuxleyV}) and obtained the following two estimations of $\Delta$:
\begin{theorem}\label{Nowak-L}
([\cite{Nowak1}, Theorem 1], [\cite{Nowak2}, Theorem 3]).
Put $D_{\max }:=\max \left(D_{1}, D_{2}, D_{3}\right)$, then for any $x \geqslant 1$, we have
\begin{equation}
\nonumber \Delta \ll x^{\varepsilon} \left(D^{\frac{1153}{5073}} D_{\max }^{\frac{474}{1691}} x^{\frac{2498}{5073}} +D^{\frac{14077}{81168}} D_{\max }^{\frac{54661}{162336}} x^{\frac{26507}{54112}} +D^{\frac{38017}{144981}} D_{\max }^{\frac{3986}{16109}} x^{\frac{71090}{144981}} +D^{\frac{856679}{2404602}} D_{\max }^{\frac{1281}{6764}} x^{\frac{2185055}{4809204}}\right) \tag{A}
\end{equation}
and
\begin{align}
\nonumber \Delta \ll &  x^{\varepsilon} \left(D^{\frac{1153}{5073}} D_{\max }^{\frac{474}{1691}} x^{\frac{2498}{5073}} +D^{\frac{3043}{12387}} D_{\max }^{\frac{1086}{4129}} x^{\frac{6086}{12387}} +D^{\frac{2719}{13362}} D_{\max }^{\frac{2719}{8908}} x^{\frac{13129}{26724}}+D^{\frac{1}{4}} D_{\max }^{\frac{1}{2}} x^{\frac{1}{4}} \right. \\
& \left. +D^{\frac{2}{7}} D_{\max }^{\frac{2}{7}} x^{\frac{3}{7}}
+D^{\frac{1}{3}} x^{\frac{1}{3}}
+D^{\frac{1}{6}} D_{\max }^{\frac{1}{2}} x^{\frac{1}{3}} +D^{\frac{328}{2719}} D_{\max }^{\frac{1}{2}} x^{\frac{2063}{5438}} +D^{\frac{205}{1422}} D_{\max }^{\frac{1}{2}} x^{\frac{253}{711}}\right). \tag{B}
\end{align}
\end{theorem}
Obviously, both (A) and (B) in Theorem~\ref{Nowak-L} imply the same bound $\Delta \ll D^{\frac{2575}{5073}} x^{\frac{2498}{5073}+\varepsilon}$. In this paper, we further combine Nowak's work with Bourgain's new bound (see \cite{Bourgain}) and get some improvement on Theorem~\ref{Nowak-L}. We will give a detailed proof of one of our results which is the following one:
\begin{theorem}\label{New-L}
Put $D_{\max }:=\max \left(D_{1}, D_{2}, D_{3}\right)$, then for any $x \geqslant 1$, we have
\begin{align}
\nonumber \Delta \ll & x^{\varepsilon} \left(D^{\frac{118}{519}} D_{\max }^{\frac{97}{346}} x^{\frac{511}{1038}} +D^{\frac{121}{692}} D_{\max }^{\frac{467}{1384}} x^{\frac{675}{1384}} +D^{\frac{56039}{213309}} D_{\max }^{\frac{69941}{284412}} x^{\frac{419257}{853236}} +D^{\frac{17936}{50343}} D_{\max }^{\frac{131}{692}} x^{\frac{91507}{201372}}\right) \\
\nonumber \ll & D^{\frac{527}{1038}} x^{\frac{511}{1038}+\varepsilon}.
\end{align}
\end{theorem}

Then by the similar arguments as in \cite{FI2004}, we can get the following result on primes in short intervals:
\begin{theorem}\label{0.4923}
Let $\chi=\chi_{D}$ denotes the real primitive character of conductor D, $x \geqslant D^{r}$ with $r=433433$ and $0.4923 \leqslant \alpha \leqslant 1$. Then we have
$$
\psi(x)-\psi(x-x^{\alpha})=x^{\alpha} \left\{1+O\left(L(1, \chi)(\log x)^{r^{r}}\right)\right\}.
$$
and
$$
\pi(x)-\pi(x-x^{\alpha})=\frac{x^{\alpha}}{\log x}\left\{1+O\left(L(1, \chi)(\log x)^{r^{r}}\right)\right\}.
$$
\end{theorem}
\begin{remark*}
We have $0.4936<\frac{39}{79}<0.4937$ and $0.4924<\frac{2500}{5077}<0.4925$.
\end{remark*}

Compare our Theorem~\ref{0.4923} with Theorem~\ref{FI0.4937}, we can find that our result is non-trivial if
\begin{equation}
L(1, \chi) \ll (\log x)^{-433433^{433433}-1}
\end{equation}
or
\begin{equation}
L(1, \chi) \ll (\log D)^{-433433^{433433}-1}
\end{equation}
if $D$ is a positive power of $x$. Unfortunately, Zhang \cite{Zhang} posited that there is no $L$-function with 
\begin{equation}
L(1, \chi) \ll (\log D)^{-2022}.
\end{equation}
\begin{corollary}
If the condition (3) holds, then there is always a prime number in the interval $(x-x^{0.4923}, x]$.
\end{corollary}

\section{Higher order derivative tests for single exponential sums}

\begin{lemma}\label{higher}
For $r \geqslant 4$ a fixed integer, and positive real parameters $M \geqslant 1$ and $T$, suppose that $F$ is a real function on some compact interval $I^{*}$ of length $M$, with $r+1$ continuous derivatives satisfying throughout
$$
F^{(j)} \asymp TM^{-j} \quad \operatorname{for} \quad j=r-2, r-1, r .
$$
Then, for every interval $I \subseteq I^{*}$,
$$
E := \sum_{m \in I} \mathrm{e}^{2 \pi \mathrm{i} F(m)} \ll M^{a_{r}} T^{b_{r}}+M^{\xi_{r}} T^{\eta_{r}}+M^{\alpha_{r}}+M^{\gamma_{r}} T^{-\delta_{r}}
$$
where
$a_{4}=\frac{1}{2}, b_{4}=\frac{13}{84}+\varepsilon, \xi_{4}=0, \eta_{4}=\frac{31}{84}+\varepsilon, \alpha_{4}=\frac{334}{411}, \gamma_{4}=1, \delta_{4}=\frac{1}{2}$,
and for every $j>4$,
$$
\begin{gathered}
a_{j}=\frac{a_{j-1}+b_{j-1}+1}{2\left(b_{j-1}+1\right)}, \quad b_{j}=\frac{b_{j-1}}{2\left(b_{j-1}+1\right)}, \\
\xi_{j}=\frac{\left(\xi_{j-1}+1\right)\left(b_{j-1}+1\right)-a_{j-1} \eta_{j-1}}{2\left(b_{j-1}+1\right)}, \quad \eta_{j}=\frac{\eta_{j-1}}{2\left(b_{j-1}+1\right)}, \quad \alpha_{j}=\frac{\alpha_{j-1}+1}{2}, \\
\gamma_{j}=\frac{a_{j-1} \delta_{j-1}+\left(b_{j-1}+1\right)\left(\gamma_{j-1}+1\right)}{2\left(b_{j-1}+1\right)}, \quad \delta_{j}=\frac{\delta_{j-1}}{2\left(b_{j-1}+1\right)}.
\end{gathered}
$$
\end{lemma}
\begin{proof}
The proof is very similar to that of [\cite{Nowak2012}, Theorem 1]. The only difference is that we use the bound $E \ll M^{\frac{1}{2}}T^{\frac{13}{84}+\varepsilon}$ in the range $T^{\frac{17}{42}+\varepsilon} \ll M \ll T^{\frac{1}{2}}$ and the bound $E \ll M^{\frac{1}{2}}T^{\frac{32}{205}+\varepsilon} \ll M^{\frac{333}{410}+\varepsilon} \ll M^{\frac{334}{411}}$ in the range $T^{\frac{1}{2}} \ll M \ll T^{\frac{181}{328}+\varepsilon}$. We remark that Bourgain and Huxley considered the exponential sums defined on different summation ranges of $m$. For this, we just need to set a function $g$ by $g(m)=f(2m)$ or $g(m/2M)=F(m/M)$ in the range $T^{\frac{17}{42}+\varepsilon} \ll M \ll T^{\frac{1}{2}}$, where the function $f$ is defined by $F(m)=Tf(m/M)$ as the same as in Nowak's papers.
\end{proof}

\begin{lemma}\label{l2}
For positive real parameters $M \geqslant 1$ and $T$, suppose that $F$ is a real function on some compact interval $I^{*}$ of length $M$, with 6 continuous derivatives satisfying throughout
$$
F^{(j)} \asymp TM^{-j} \quad \operatorname{for} \quad j=3,4,5 .
$$
Then, for every interval $I \subseteq I^{*}$,
$$
E \ll M^{\frac{139}{194}} T^{\frac{13}{194}+\varepsilon}+M^{\frac{163}{388}} T^{\frac{31}{194}+\varepsilon}+M^{\frac{745}{822}}+M^{\frac{215}{194}} T^{-\frac{21}{97}}.
$$
\end{lemma}
\begin{proof}
This is the case $r=5$ of our Lemma~\ref{higher}.
\end{proof}

\section{A problem considered by Friedlander, Iwaniec and Nowak}
We follow in all essentials the argument of \cite{Nowak1} and using our Lemma~\ref{l2}. We write $e(w):=\mathrm{e}^{2 \pi \mathrm{i} w}$ and start from the estimate
\begin{equation}
\Delta \ll\left|D^{\frac{1}{6}} x^{\frac{1}{3}} \sum_{n_{1} n_{2} n_{3} \leqslant N} \frac{\bar{\chi}_{1}\left(n_{1}\right) \bar{\chi}_{2}\left(n_{2}\right) \bar{\chi}_{3}\left(n_{3}\right)}{\left(n_{1} n_{2} n_{3}\right)^{\frac{2}{3}}} e\left( \pm 3\left(\frac{n_{1} n_{2} n_{3}x}{D}\right)^{\frac{1}{3}}\right)\right|+x^{\varepsilon}\left(\frac{D x^{2}}{N}\right)^{\frac{1}{3}}
\end{equation}
where the sign $\pm$ is chosen so that the modulus involved becomes maximal. This result is immediate from [\cite{FI2005}, Proposition 3.2]. The exponential sum here can be split up into $O\left(\log ^{3}x\right)$ subsums
$$
S\left(N_{1}, N_{2}, N_{3}\right):=\sum_{\substack{n_{j} \sim N_{j} \\ j=1,2,3}} \frac{\bar{\chi}_{1}\left(n_{1}\right) \bar{\chi}_{2}\left(n_{2}\right) \bar{\chi}_{3}\left(n_{3}\right)}{\left(n_{1} n_{2} n_{3}\right)^{\frac{2}{3}}}  e\left( \pm 3\left(\frac{n_{1} n_{2} n_{3}x}{D}\right)^{\frac{1}{3}}\right)
$$
with $N_{1} N_{2} N_{3} \leqslant N$ throughout. Without loss of generality we assume that $N_{1} \leqslant N_{2} \leqslant N_{3}$, and put $P=N_{1} N_{2} N_{3}$. Then obviously
\begin{equation}
S\left(N_{1}, N_{2}, N_{3}\right) \ll P^{-\frac{2}{3}} \sum_{\substack{n_{j} \sim N_{j} \\ j=1,2}}\left|\sum_{n_{3} \sim N_{3}} \bar{\chi}_{3}\left(n_{3}\right) e\left( \pm 3\left(\frac{n_{1} n_{2} n_{3}x}{D}\right)^{\frac{1}{3}}\right)\right|.
\end{equation}
Here integration by parts has been used with respect to $n_{3}$, and trivial estimation with respect to $n_{1}, n_{2}$. For each fixed pair $\left(n_{1}, n_{2}\right)$, the range for $n_{3}$ is chosen in such a way that the absolute value on the right hand side becomes maximal.

The next step is to express the character $\bar{\chi}_{3}$ by means of the Gauss sums $G\left(m, \bar{\chi}_{3}\right)$:
$$
\bar{\chi}_{3}\left(n_{3}\right)=\frac{1}{D_{3}} \sum_{m=1}^{D_{3}} G\left(m, \bar{\chi}_{3}\right) e\left(-\frac{m n_{3}}{D_{3}}\right), \quad G\left(m, \bar{\chi}_{3}\right):=\sum_{k=1}^{D_{3}} \bar{\chi}_{3}(k) e\left(\frac{k m}{D_{3}}\right)
$$
Since $\left|G\left(m, \bar{\chi}_{3}\right)\right| \leqslant D_{3}^{\frac{1}{2}} \leqslant D_{\max }^{\frac{1}{2}}$, it follows that for each fixed $\left(n_{1}, n_{2}\right)$, with $n_{1} \sim N_{1}$, $n_{2} \sim N_{2}$,
$$
\sum_{n_{3} \sim N_{3}} \bar{\chi}_{3}\left(n_{3}\right) e\left( \pm 3\left(\frac{n_{1} n_{2} n_{3}x}{D}\right)^{\frac{1}{3}}\right) \ll D_{\max }^{\frac{1}{2}} \max _{m=1, \ldots, D_{3}}\left|\sum_{n_{3} \sim N_{3}} e\left(3\left(\frac{n_{1} n_{2} n_{3}x}{D}\right)^{\frac{1}{3}} \mp \frac{m n_{3}}{D_{3}}\right)\right| .
$$
The last exponential sums will now be bounded by Lemma~\ref{l2}. The conditions of Lemma~\ref{l2} are satisfied with
$$
M=N_{3}, \quad T=\left(\frac{P x}{D}\right)^{\frac{1}{3}}.
$$

Therefore, by Lemma~\ref{l2} we have
\begin{align}
\nonumber & \sum_{n_{3} \sim N_{3}} \bar{\chi}_{3}\left(n_{3}\right) e\left(3\left(\frac{n_{1} n_{2} n_{3}x}{D}\right)^{\frac{1}{3}}\right) \\
\ll & D_{\max }^{\frac{1}{2}} \left(N_{3}^{\frac{139}{194}}\left(\left(\frac{P x}{D}\right)^{\frac{1}{3}}\right)^{\frac{13}{194}+\varepsilon}+N_{3}^{\frac{163}{388}}\left(\left(\frac{P x}{D}\right)^{\frac{1}{3}}\right)^{\frac{31}{194}+\varepsilon} +N_{3}^{\frac{745}{822}}+N_{3}^{\frac{215}{194}}\left(\left(\frac{P x}{D}\right)^{\frac{1}{3}}\right)^{-\frac{21}{97}}\right).
\end{align}
By (7)--(8) we get that
\begin{align}
\nonumber & D^{\frac{1}{6}} x^{\frac{1}{3}} S\left(N_{1}, N_{2}, N_{3}\right) \\
\nonumber \ll& D^{\frac{1}{6}} x^{\frac{1}{3}} P^{\frac{1}{3}} D_{\max }^{\frac{1}{2}} \left(D^{-\frac{13}{582}} N_{3}^{-\frac{55}{194}} P^{\frac{13}{582}} x^{\frac{13}{582}}+D^{-\frac{31}{582}} N_{3}^{-\frac{225}{388}} P^{\frac{31}{582}} x^{\frac{31}{582}}+N_{3}^{-\frac{77}{822}}+D^{\frac{7}{97}} N_{3}^{\frac{21}{194}} P^{-\frac{7}{97}} x^{-\frac{7}{97}}\right) \\
\ll& D_{\max }^{\frac{1}{2}} \left(D^{\frac{14}{97}} N_{3}^{-\frac{55}{194}} P^{\frac{69}{194}} x^{\frac{69}{194}}+D^{\frac{11}{97}} N_{3}^{-\frac{225}{388}} P^{\frac{75}{194}} x^{\frac{75}{194}}+D^{\frac{1}{6}} N_{3}^{-\frac{77}{822}} P^{\frac{1}{3}} x^{\frac{1}{3}}+D^{\frac{139}{582}} N_{3}^{\frac{21}{194}} P^{\frac{76}{291}} x^{\frac{76}{291}}\right).
\end{align}
By the trivial facts $P^{\frac{1}{3}} \leqslant N_{3} \leqslant N$ to get rid of $N_{3}$ and $P \ll N$, we have
\begin{align}
\Delta \ll& \left(\frac{D x^{2}}{N}\right)^{\frac{1}{3}}+D_{\max }^{\frac{1}{2}}\left(D^{\frac{14}{97}} N^{\frac{76}{291}} x^{\frac{69}{194}} +D^{\frac{11}{97}} N^{\frac{75}{388}} x^{\frac{75}{194}}+D^{\frac{1}{6}} N^{\frac{745}{2466}} x^{\frac{1}{3}} +D^{\frac{139}{582}} N^{\frac{215}{582}} x^{\frac{76}{291}}\right).
\end{align}
Finally, by choosing
$$
N=D^{\frac{55}{173}} D_{\max }^{-\frac{291}{346}} x^{\frac{181}{346}},
$$
we obtain the bound
\begin{align}
\nonumber \Delta \ll & x^{\varepsilon} \left(D^{\frac{118}{519}} D_{\max }^{\frac{97}{346}} x^{\frac{511}{1038}} +D^{\frac{121}{692}} D_{\max }^{\frac{467}{1384}} x^{\frac{675}{1384}} +D^{\frac{56039}{213309}} D_{\max }^{\frac{69941}{284412}} x^{\frac{419257}{853236}} +D^{\frac{17936}{50343}} D_{\max }^{\frac{131}{692}} x^{\frac{91507}{201372}}\right) \\
\ll & D^{\frac{527}{1038}} x^{\frac{511}{1038}+\varepsilon},
\end{align}
which is our Theorem~\ref{New-L}, improves Nowak's Theorem~\ref{Nowak-L} (we have $0.4924 < \frac{2498}{5073} < 0.4925$ and $0.4922 < \frac{511}{1038} < 0.4923$). 

\section{A new application: Primes in Short Intervals}

In his articles \cite{Nowak2012}--\cite{Nowak2}, Nowak considered several problems concerning products of $L$-series and pointed that his method is more powerful in the case of exponential sums depending on several parameters. For exponential sums depending on one parameter only, multiple sum estimations are often more powerful. (see [\cite{Nowak2}, Concluding Remark]). Clearly our new estimation can improve these results. Now we follow Friedlander and Iwaniec directly and show that our Theorem~\ref{New-L} can be used to extend the range of $\alpha$.
We first introduce some arithmetic functions:
$$
\lambda(n):=\sum_{a d=n} \chi(d), \quad \nu(m):=\sum_{a d=m} \mu(a) \mu(d) \chi(d), \quad \rho(n):=\sum_{l m=n} \lambda(m),
$$
and a close relative of $\lambda(d)$, namely
$$
\lambda^{\prime}(d):=\sum_{k l=d} \chi(k) \log l
$$
which has the following properties:
$$
\lambda^{\prime}(d) = \sum_{b c=d} \lambda(b) \Lambda(c), \quad \Lambda(n) = \sum_{d m=n} \lambda^{\prime}(d) \nu(m), \quad 0 \leqslant \lambda^{\prime}(d) \leqslant \tau(d) \log d .
$$

In the classical divisor problem, one need to get an asymptotic formula for the functions:
$$
\mathcal{D}\left(x ; \lambda\right) :=\sum_{n \leqslant x} \lambda(n), \quad \mathcal{D}\left(x ; \lambda^{\prime}\right) :=\sum_{n \leqslant x} \lambda^{\prime}(n)
$$
and we also need the asymptotic formula for
$$
\mathcal{D}\left(x ; \rho\right):=\sum_{n \leqslant x} \rho(n).
$$
\begin{lemma}\label{l41}
For $x \geqslant 1$ we have
\begin{align}
\mathcal{D}\left(x ; \lambda\right)=& L(1, \chi) x+O\left(D^{\frac{1}{3}} x^{\frac{1}{3}+\varepsilon}\right), \\
\mathcal{D}\left(x ; \lambda^{\prime}\right)=& L(1, \chi) x \log x+\left(L^{\prime}(1, \chi)-L(1, \chi)\right) x +O\left(D^{\frac{1}{3}} x^{\frac{1}{3}+\varepsilon}\right), \\
\mathcal{D}\left(x ; \rho\right)=& L(1, \chi) x \log x+\left(L^{\prime}(1, \chi)+(2 \gamma-1) L(1, \chi)\right) x +O\left(D^{\frac{527}{1038}} x^{\frac{511}{1038}+\varepsilon}\right),
\end{align}
where $\gamma$ is the Euler constant and the implied constant depends only on $\varepsilon$.
\end{lemma}
\begin{proof}
These can be proved by similar arguments as in \cite{FI2005}, and the only difference is that we use our Theorem~\ref{New-L} instead of Theorem~\ref{FI-L} in order to prove the formula (14).
\end{proof}

Now we split the functions $\rho=\rho^{*}+\rho_{*}$, $\Lambda=\Lambda^{*}+\Lambda_{*}$ and further define
$$
\rho^{*}(n):=\sum_{\substack{l m=n \\ m \leqslant D^{2}}} \lambda(m), \quad \rho_{*}(n):=\sum_{\substack{l m=n \\ m > D^{2}}} \lambda(m),
$$
$$
\Lambda^{*}(n):=\sum_{\substack{d m=n \\ m \leqslant D^{2}}} \lambda^{\prime}(d) \nu(m), \quad \Lambda_{*}(n):=\sum_{\substack{d m=n \\ m > D^{2}}} \lambda^{\prime}(d) \nu(m),
$$
$$
\psi^{*}(x):=\sum_{n \leqslant x} \Lambda^{*}(n), \quad \psi_{*}(x):=\sum_{n \leqslant x} \Lambda_{*}(n).
$$
\begin{lemma}\label{l42}
[\cite{FI2004}, Corollary 4.2]. For $D^{\frac{5}{2}} x^{\frac{3}{8}}<y \leqslant x$ we have
\begin{equation}
\psi^{*}(x)-\psi^{*}(x-y)=y+O\left(L(1, \chi) y(\log x)^{9}\right),
\end{equation}
where the implied constant is absolute.
\end{lemma}
By Lemma~\ref{l41} and similar arguments as in \cite{FI2004}, for $D^{2} x^{\frac{492293}{1000000}}<y \leqslant x$ we obtain the following asymptotic formula
\begin{equation}
\mathcal{D}\left(x ; \rho_{*}\right)-\mathcal{D}\left(x-y ; \rho_{*}\right)=L(1, \chi) y\{\log x+O(1)\}.
\end{equation}

Now by (16) and
\begin{equation}
\left|\psi_{*}(x)-\psi_{*}(x-y)\right| \leqslant (\log x) \sum_{\delta \leqslant \Delta}(2 \tau(\delta))^{A}\left[\mathcal{D}\left(x \delta^{-1} ; \rho_{*}\right)-\mathcal{D}\left((x-y) \delta^{-1} ; \rho_{*}\right)\right]
\end{equation}
with the bound
\begin{equation}
\sum_{\delta \leqslant \Delta} \tau(\delta)^{A} \delta^{-1} \ll(\log \Delta)^{2^{A}+1}
\end{equation}
where $\Delta=x^{\frac{1}{r}}, A=\frac{r \log r}{\log 2}$ and $r$ is a positive integer, we have

\begin{lemma}\label{l43}
For $D^{\frac{5}{2}} x^{\frac{492293}{1000000}+\frac{507707}{1000000 r}}<y \leqslant x$ with $r \geqslant 1$, we have
\begin{equation}
\psi_{*}(x)-\psi_{*}(x-y) =O\left(L(1, \chi) y(\log x)^{r^{r}+3}\right),
\end{equation}
where the implied constant depends only on $r$.
\end{lemma}

Adding (19) to (15) we find that
\begin{equation}
\psi(x)-\psi(x-y) =O\left(L(1, \chi) y(\log x)^{r^{r}+3}\right),
\end{equation}
subject to the conditions of Lemma~\ref{l43}. We require $D \leqslant x^{\frac{1}{r}}$ and then choose a number $\theta$ with $\frac{492293}{1000000}<\theta<\frac{1}{2}$ satisfies
$$
\frac{5}{2r}+\frac{492293}{1000000}+\frac{507707}{1000000 r} \leqslant \theta.
$$
By choosing $\theta=0.4923$ and $r=433433$, we complete the proof of Theorem~\ref{0.4923}. We remark that the range $0.4923 \leqslant \alpha \leqslant 1$ is rather near to the limit obtained by this method. In his preprint \cite{Merikoski}, Merikoski mentioned that some sieve arguments can be used to this problem.

\bibliographystyle{plain}
\bibliography{bib}
\end{document}